\documentclass[12pt,a4paper,reqno]{amsart}
\usepackage{amsmath}
\usepackage{amssymb,latexsym}
\usepackage[dvips]{graphicx}
\usepackage{color}

\newcommand{\R}{\mathbb R}

\newtheorem{theorem}{Theorem} [section]
\newtheorem{lemma}{Lemma} [section]

\newtheorem{corollary}{Corollary} [section]
\newtheorem{definition}{Definition} [section]

\newtheorem{remark}{Remark}[section]

\let\ssection=\section\renewcommand{\section}{\setcounter{equation}{0}\ssection}
\begin{document}

\title [Universal inequalities]{Universal inequalities for the eigenvalues of a power of the Laplace operator}
\author{Sa\"{\i}d Ilias and Ola Makhoul}
\date{09 novembre 2009}

\address{S. Ilias, O. Makhoul: Universit\'e Fran\c{c}ois rabelais de Tours, Laboratoire de Math\'ematiques
et Physique Th\'eorique, UMR-CNRS 6083, Parc de Grandmont, 37200
Tours, France} \email{ilias@univ-tours.fr, ola.makhoul@lmpt.univ-tours.fr}

\keywords{eigenvalues, Laplacian, polyharmonic operator, biharmonic operator, clamped plate, Payne-Polya-Weinberger inequality, Hile-Protter inequality, Yang inequality, universal inequalities, commutators, Kohn Laplacian, Heisenberg group}
\subjclass[2000]{35P15;58C40}

\begin{abstract}
In this paper, we obtain a new abstract formula relating
eigenvalues of a self-adjoint operator to two families of symmetric
and skew-symmetric operators and their commutators. This formula generalizes earlier ones obtained by Harrell, Stubbe, Hook, Ashbaugh, Hermi, Levitin and Parnovski. We also show how one
can use this abstract formulation both for giving different and simpler
proofs for all the known results obtained for the eigenvalues of a power of the Laplace operator (i.e. the Dirichlet Laplacian, the clamped plate problem for the bilaplacian and more generally for the polyharmonic problem on a bounded Euclidean domain) and to obtain new ones. In a last paragraph, we derive new bounds for eigenvalues of any power of the Kohn Laplacian on the Heisenberg group.
\end{abstract}

\maketitle
\section{Introduction}
Let $\Omega$  be a bounded domain of an n-dimensional Euclidean space $\R^n$ and consider the following eigenvalue problem for the polyharmonic operator :
\begin{equation}\label {pol}
\begin{cases}
(-\Delta)^l u  = \lambda u \; \;\text{in} \,\, \Omega, \\
u = \displaystyle{\frac{\partial u}{\partial \nu}=\dots=\frac{\partial^{l-1}
u}{\partial \nu^{l-1}}=0 \;\;\text{on} \,\, \partial \Omega, }
\end{cases}
\end{equation}
where $\Delta$ is the Laplace operator and $\nu$ is the outward unit normal.\\
It is known that this eigenvalue problem has a discrete spectrum, 
$$ 0 < \lambda_1 < \lambda_2 \leq \dots \leq \lambda_k \leq
\ldots \rightarrow +\infty $$  
In this paper we will be interested in "Universal"(i.e. not depending on the domain) inequalities for the eigenvalues of such a polyharmonic  problem and especially we will show how to derive them from a general abstract algebraic formula in the spirit of the work of  Harrell, Stubbe, Ashbaugh and Hermi.\\
\indent Let us begin by giving a short and non-exhaustive presentation of the known results in this field.

The first result concerns the Dirichlet Laplacian (i.e. when $l=1$). In this case, Polya, Payne and Weinberger (henceforth PPW) proved in 1955 the following bound (see \cite{PPW} for dimension 2 and \cite{Thom} for all dimensions), for $k=1,2,\dots$
\begin{equation}\label{ppw}
\displaystyle \lambda_{k+1}-\lambda_k \leq \frac{4}{nk} \sum_{i=1}^k
\lambda_i,
\end{equation}
This result was improved in 1980 by Hile and Protter \cite{HileProt} (henceforth HP) who showed that, for $k=1,2,\ldots$
\begin{equation}
\displaystyle \frac{nk}{4} \leq \sum_{i=1}^k
\frac{\lambda_i}{\lambda_{k+1}-\lambda_i}.
\end{equation}
In 1991, H.C.Yang (see \cite{Yang.HC} and more recently
\cite{ChengYang1}) proved
\begin{equation}\label{1}
\displaystyle \sum_{i=1}^k (\lambda_{k+1}-\lambda_i)^2 \leq
\frac{4}{n} \sum_{i=1}^k \lambda_i(\lambda_{k+1}-\lambda_i),
\end{equation}
which is, until now, the best improvement of the PPW inequality.
From inequality (\ref{1}), we can infer a weaker form
\begin{equation}\label{yang2}
  \lambda_{k+1} \le (1+\frac{4}{n})\frac{1}{k} \left(\sum_{i=1}^k \lambda_i\right).
\end{equation}
We shall refer to inequality (\ref{1}) as Yang's first inequality (or simply Yang inequality) and to inequality (\ref{yang2}) as Yang's second inequality.\\
The comparison of all these inequalities (see \cite{Ashb2}) can be summarized in
$$ \text{Yang}\;1 \Longrightarrow  \text{Yang}\;2 \Longrightarrow \text{HP}  \Longrightarrow \text{PPW}  $$

When $l=2$, the eigenvalue problem (\ref {pol}) for the bilaplacian is the clamped plate problem. In the same paper as before \cite {PPW}, Polya, Payne and Weinberger proved the following analog of the formula (\ref{ppw})
\begin{equation}\label{2}
\displaystyle \lambda_{k+1}-\lambda_k \leq \frac{8(n+2)}{n^2 k}
\sum_{i=1}^k \lambda_i.
\end{equation}
And as was noticed by Ashbaugh (see \cite{Ashb} inequality (3.56)), there is a better inequality which was implicit in the PPW work, 
\begin{equation*}
 \lambda_{k+1}-\lambda_k \leq \frac{8(n+2)}{n^2k^2}\bigg( \sum_{i=1}^k \lambda_i^{\frac{1}{2}} \bigg)^2.
\end{equation*}
In 1984, Hile and Yeh \cite{HileYeh} extended the approach used for the Laplacian in \cite{HileProt} and proved the sharpest bound
\begin{equation}\label{2'}
\displaystyle \frac{n^2 k^{\frac{3}{2}}}{8(n+2)} \leq \bigg(\sum_{i=1}^k
\frac{\lambda_i^{\frac{1}{2}}}{\lambda_{k+1}-\lambda_i}\bigg)
\bigg(\sum_{i=1}^k \lambda_i \bigg)^{\frac{1}{2}}.
\end{equation}
Then in 1990, Hook \cite{Hook1}, Chen and Qian \cite{ChenQian3} proved
independently the following stronger inequality which was again implicit in the work of Hile and Yeh (see also \cite{Ashb}, \cite{ChenQian1}, \cite{ChenQian2} and \cite{ChenQian4})
\begin{equation}\label{3}
\displaystyle \frac{n^2 k^2}{8(n+2)} \leq \bigg( \sum_{i=1}^k
\frac{\lambda_i^{\frac{1}{2}}}{\lambda_{k+1}-\lambda_i}\bigg)
\bigg(\sum_{i=1}^k \lambda_i^{\frac{1}{2}} \bigg).
\end{equation}
Using Chebyshev inequality, Ashbaugh (see \cite{Ashb} inequality (3.60)) deduces from the preceding inequality (\ref{3}), the following HP version which is weaker and more esthetically appealing,
\begin{equation}\label{3'}
\displaystyle \frac{n^2 k}{8(n+2)} \leq \sum_{i=1}^k
\frac{\lambda_i}{\lambda_{k+1}-\lambda_i}.
\end{equation}
Recently, Cheng and Yang \cite{ChengYang} established the following Yang version
\begin{equation}\label{4}
\sum_{i=1}^{k}(\lambda_{k+1}-\lambda_i )\leq
\bigg[ \frac{8(n+2)}{n^2} \bigg]^{\frac{1}{2}}
\sum_{i=1}^k \Big[ \lambda_i(\lambda_{k+1}-\lambda_i)
\Big]^{\frac{1}{2}}.
\end{equation}
For any $l$, the PPW inequality is given by 
\begin{equation*}
\displaystyle \lambda_{k+1}-\lambda_k \leq \frac{4l(2l+n-2)}{n^2
k^2} \bigg(\sum_{i=1}^k \lambda_i^{\frac{1}{l}}\bigg) \bigg(\sum_{i=1}^k
\lambda_i^{\frac{l-1}{l}} \bigg).
\end{equation*}
Its HP improvement was proved independently by Hook
\cite{Hook1} and Chen and Qian \cite{ChenQian3}, it reads
\begin{equation}\label{7}
\displaystyle \frac{n^2 k^2}{4l(2l+n-2)}\leq \sum_{i=1}^k
\frac{\lambda_i^{\frac{1}{l}}}{\lambda_{k+1}-\lambda_i} \sum_{i=1}^k
\lambda_i^{\frac{l-1}{l}}.
\end{equation}
As in the case $l=2$ (inequality (\ref{3'})), this reduces to the weaker form 
\begin{equation}\label{7'}
\displaystyle \frac{n^2 k}{4l(2l+n-2)}\leq \sum_{i=1}^k
\frac{\lambda_i}{\lambda_{k+1}-\lambda_i}.
\end{equation}
In 2007, Wu and Cao \cite{WuCao} generalized the
inequality (\ref{4}) of Cheng and Yang to the polyharmonic problem
and obtained 
\begin{align}\label{WuCao}
{}&\sum_{i=1}^k (\lambda_{k+1}-\lambda_i) \leq \\
\frac{1}{n}\left(4l(n+2l-2)\right)^{\frac{1}{2}}&\bigg( \sum_{i=1}^k (\lambda_{k+1}-\lambda_i)^{\frac{1}{2}} \lambda_i^{\frac{l-1}{l}}\bigg)^{\frac{1}{2}} 
\bigg(\sum_{i=1}^k(\lambda_{k+1}-\lambda_i)^{\frac{1}{2}}\lambda_i^{\frac{1}{l}}\bigg)^{\frac{1}{2}}.\nonumber
\end{align}
This inequality is sharper than inequality (\ref{7})(see
\cite{WuCao}).\\
 Very recently, Cheng, Ichikawa and Mametsuka \cite{Cheng-Ichik-Mamet} derived the following Yang type inequality for the polyharmonic operator (i.e. such that for $l=1$,
we have the Yang inequality (\ref{1}))
\begin{equation}\label{8'}
\sum_{i=1}^k(\lambda_{k+1}-\lambda_i)^2 \leq
\frac{4l(2l+n-2)}{n^2}\sum_{i=1}^k(\lambda_{k+1}-\lambda_i)\lambda_i.
\end{equation}
\indent All the classical proofs of these inequalities are based on tricky
and careful choices of trial functions. For a more comprehensive and general approach, it is important to see if all these inequalities can be deduced using purely algebraic arguments involving eigenvalues and eigenfunctions of an abstract self-adjoint operator acting on a Hilbert space. In the case of the Laplacian (i.e. $l=1$), this was done by Harrell \cite{Harl1,Harl2}, Harrell and Michel \cite{HarlMichel1, harlMichel2}, Harrell and Stubbe \cite{HarlStub}, and Ashbaugh and Hermi \cite{AshbHer1}.\\
For the polyharmonic problem (i.e. general $l$), Hook \cite{Hook1} generalized the argument of Hile and Protter \cite{HileProt} in an abstract setting. Later, this abstract formulation of Hook was simplified and improved by Ashbaugh and Hermi \cite{AshbHer2}. In fact, they obtained the following
inequality relating eigenvalues of a self-adjoint
operator $A$, to two families of symmetric operators $B_p's$,
skew-symmetric operators $T_p's$, $p=1,\dots,n$ and their commutators (for a precise statement with detailed assumptions, see Theorem 2.2 of \cite{AshbHer2}),
\begin{equation}\label{Ashb-Her}
\frac{1}{4} \frac{\Big(\sum_{i=1}^k \sum_{p=1}^n \langle
[B_p,T_p]u_i,u_i \rangle \Big)^2}{\sum_{i=1}^k \sum_{p=1}^n \langle
[A,B_p]u_i,B_p u_i \rangle} \leq \sum_{i=1}^k \sum_{p=1}^n \frac
{\langle T_p u_i,T_p u_i \rangle}{\lambda_{k+1}-\lambda_i}.
\end{equation}
But this abstract inequality, as was observed by Ashbaugh and Hermi in the end of the third paragraph of their article \cite{AshbHer2},  could not recover
more than the HP version of the universal inequalities (i.e. inequalities (\ref{7}) and (\ref{7'})). \\
The main goal of the present paper is
to prove the following abstract inequality (with the same assumptions as those for the Ashbaugh-Hermi inequality (\ref{Ashb-Her}))which generalizes (\ref{Ashb-Her}) and fills this gap 
\begin{align} \label{main formula}
\bigg(\sum_{i=1}^k
\sum_{p=1}^n &f(\lambda_i)\langle[T_p,
B_p]u_i,u_i\rangle \bigg)^2 \\
\le 4 \bigg(\sum_{i=1}^k\sum_{p=1}^n g(\lambda_i)& \langle
[A,B_p]u_i, B_p u_i\rangle \bigg) \bigg(\sum_{i=1}^k
\sum_{p=1}^n \frac{\big(f(\lambda_i)\big)^2}{g(\lambda_i)(\lambda_{k+1}-\lambda_i)}\|T_p u_i\|^2\bigg),\nonumber
\end{align}
where $f$ and $g$ are two functions satisfying some functional conditions (see Definition (\ref{def1})). The family of such couples of functions is large and particular choices for $f$ and $g$ give many of the known universal inequalities. For instance, in the case of the polyharmonic problem, if we take $f(x)=g(x)=(\lambda_{k+1}-x)^{2}$, then we obtain the Yang type inequality (\ref{8'}) proved by Cheng, Ichikawa and Mametsuka and when we take $f(x)=(g(x))^{2}=(\lambda_{k+1}-x)$, we obtain the Wu-Cao inequality (\ref{WuCao}). \\
On the other hand, we observe that by taking $T_{p}=[A,B_p]$, we obtain the following new formula (see Corollary \ref{first corollary}), where only one family of symmetric operators $B_{p}$ is needed

\begin{align}
\bigg[\sum_{i=1}^k \sum_{p=1}^n &f(\lambda_i)\langle[A,
B_p]u_i,B_{p}u_i\rangle\bigg]^2 \nonumber \\
{}\leq \bigg[\sum_{i=1}^k  \sum_{p=1}^n
g(\lambda_i)&\langle [A,B_p]u_i, B_p
u_i\rangle \bigg] \bigg[\sum_{i=1}^k \sum_{p=1}^n
\frac{\Big(f(\lambda_i)\Big)^2}{g(\lambda_i)(\lambda_{k+1}-\lambda_i)} \|\left[A,B_{p}\right]
u_i\|^2\bigg].
\end{align}
Using this last inequality, with particular choices of $f$ and $g$ as before, one can recover many of the known universal inequalities for eigenvalues of Laplace or Schr\"odinger operators. \\
\indent In the last section of this paper, we show how one can use the
inequality (\ref{main formula}) to derive new universal bounds, of Yang type,
for eigenvalues of the Kohn Laplacian on the Heisenberg group, with
any order. These bounds are stronger than the earlier bounds obtained by Niu and
Zhang in \cite{Niu-Zhang}. 
\section {\textbf{The abstract formulation}} 

Before stating the main result of this section, we introduce a special family of couples of functions which will play an important role in our formulation.
\begin{definition}\label{def1}
Let $\lambda >0$. A couple $(f,g)$ of functions defined on $]0,\lambda[$ belongs to $\Im_{\lambda}$ provided that  
\begin{itemize}
 \item[1.] $f$ and $g$ are positive,
\item[2.] $f$ and $g$ satisfy the following condition,\\ 
 for any $x,\, y \in ]0,\lambda[$ such that $x \neq y$,
\begin{equation}\label{cond}
\Big(\frac{f(x)-f(y)}{x-y}\Big)^2+\Big(\frac{\big(f(x)\big)^2}{g(x)(\lambda-x)}+\frac{\big(f(y)\big)^2}{g(y)(\lambda-y)}\Big)\Big(\frac{g(x)-g(y)}{x-y}\Big)\le 0.
\end{equation}
\end{itemize}
\end{definition}
A direct consequence of our definition is that $g$ must be nonincreasing.\\
If we multiply $f$ and $g$ of $\Im_{\lambda}$ by positive constants the resulting functions are also in $\Im_{\lambda}$. In the case where $f$ and $g$ are differentiable, one can easily deduce from (\ref{cond}) the following necessary condition:
\begin{equation*}
 \bigg[\big(\ln{f(x)}\big)'\bigg]^2 \le \frac{-2}{\lambda-x}\big(\ln{g(x)}\big)'.
\end{equation*}
This last condition helps us to find many couples $(f,g)$ satisfying the conditions 1) and 2) above. Among them, we mention
$\left\{\Big(1, (\lambda-x)^{\alpha}\Big)\, / \, \alpha \ge 0\right\}$,\; $\left\{\Big((\lambda-x),(\lambda-x)^{\beta}\Big)\, / \, \beta\ge \frac{1}{2}\right\}$, $\left\{\Big((\lambda-x)^{\delta},(\lambda-x)^{\delta}\Big)\, / \, 0<\delta \le 2\right\}$.\\ 
and  $\left\{\Big((\lambda-x)^{\alpha},(\lambda-x)^{\beta}\Big)\, / \, \alpha<0,\,1 \le \beta,\, {\rm and}\, \alpha^{2} \le \beta \right\}$ .\\

\indent Let $\mathcal{H}$ be a
complex Hilbert space with scalar product $\langle .,. \rangle$ and
corresponding norm $\|.\|$. For any two operators $A$ and $B$, we
denote by $[A,B]$ their commutator, defined by $[A,B]=AB-BA$. 
\begin{theorem}\label{first theorem}
Let $A$ : $\mathcal{D}\subset \mathcal{H}\longrightarrow
\mathcal{H}$ be a self-adjoint operator defined on a dense domain
$\mathcal{D}$, which is semibounded below and has a discrete
spectrum $\lambda_1 \leq \lambda_2\leq \lambda_3...$. 
Let $\{T_p:\mathcal{D}\longrightarrow \mathcal{H}\}_{p=1}^n$ be a collection of
skew-symmetric operators, and $\{B_p:T_p(\mathcal{D})\longrightarrow
\mathcal{H}\}_{p=1}^n$ be a collection of symmetric operators,
leaving $\mathcal{D}$ invariant. We denote by
$\left\{u_i\right\}_{i=1}^{\infty}$ a basis of orthonormal
eigenvectors of $A$, $u_{i}$ corresponding to $\lambda_i$. Let $k\ge 1$ and assume that $\lambda_{k+1}>\lambda_k$. 
Then, for any $(f,g)$ in $\Im_{\lambda_{k+1}}$  
\begin{align}\label{a}
\bigg( \sum_{i=1}^k
\sum_{p=1}^n &f(\lambda_i)\langle[T_p,
B_p]u_i,u_i\rangle \bigg)^2 \\
\le 4 \bigg(\sum_{i=1}^k\sum_{p=1}^n g(\lambda_i)& \langle
[A,B_p]u_i, B_p u_i\rangle \bigg) \bigg(\sum_{i=1}^k
\sum_{p=1}^n \frac{\big(f(\lambda_i)\big)^2}{g(\lambda_i)(\lambda_{k+1}-\lambda_i)}\|T_p u_i\|^2\bigg).\nonumber
\end{align}
\end{theorem}
\begin{proof}[Proof of Theorem \ref{first theorem}]
For each $i$, we consider the vectors $\phi_i^p$, given by 
$$\phi_i^p=B_pu_i-\sum_{j=1}^k a_{ij}^p u_j$$
 where $a_{ij}^p := \langle B_pu_i,u_j\rangle$, $p=1,...,n$. We have
\begin{equation}\label{b'}
\langle \phi_i^p, u_j \rangle =0,
\end{equation}
 for all  $j=1,...,k$. Taking $\phi_i^p$ as a trial vector in the
Rayleigh-Ritz ratio, we obtain
\begin{equation}\label{b}
\lambda_{k+1} \leq \frac{\langle A\phi_i^p,\phi_i^p \rangle}{\langle
\phi_i^p,\phi_i^p \rangle}.\end{equation}
Since $B_p$ is symmetric, for all $p=1,...,n$, we have
$a_{ij}^p=\overline{a_{ji}^p}$. Moreover, using the orthogonality conditions (\ref{b'}), we obtain
\begin{align}\label{c}
 \|\phi_i^p\|^2 & =\displaystyle \langle \phi_i^p,B_p u_i-\sum_{j=1}^k a_{ij}^p u_j \rangle= \langle \phi_i^p ,B_p u_i \rangle\nonumber\\
 & = \displaystyle \langle B_pu_i-\sum_{j=1}^k a_{ij}^p u_j,B_p u_i \rangle= \displaystyle \|B_pu_i\|^2-\sum_{j=1}^k a_{ij}^p \langle B_p u_i,u_j \rangle\nonumber\\
 & = \displaystyle \|B_p u_i\|^2-\sum_{j=1}^k\left|a_{ij}^p\right|^2
\end{align}
and
\begin{align}\label{d}
{}  \langle A\phi_i^p,\phi_i^p \rangle  & =  \displaystyle \langle AB_pu_i-\sum_{j=1}^k \lambda_j a_{ij}^p u_j, \phi_i^p \rangle \nonumber\\
{}  & = \langle AB_p u_i,\phi_i^p \rangle \nonumber\\
{}  & = \displaystyle\langle AB_p u_i, B_p u_i \rangle-\sum_{j=1}^k \overline{a_{ij}^p} \langle AB_p u_i,u_j \rangle \nonumber\\
{}  & = \displaystyle \langle [A,B_p]u_i,B_pu_i\rangle + \langle B_pAu_i,B_pu_i \rangle -\sum_{j=1}^k \lambda_j \left|a_{ij}^p\right|^2 \nonumber\\
{}  & = \displaystyle \langle [A,B_p]u_i,B_p u_i \rangle
+\lambda_i\| B_p u_i\| ^2 - \sum _{j=1}^k \lambda_j \left|a_{ij}^p\right|^2.
\end{align}
Hence, inequality (\ref{b}) reduces to
\begin{equation} \label{e}\displaystyle\lambda_{k+1}\|\phi_i^p\|^2  \leq
\displaystyle\langle[A,B_p]u_i, B_p u_i \rangle + \lambda_i\|B_p
     u_i\|^2- \sum_{j=1}^k \lambda_j \left|a_{ij}^p\right|^2.
\end{equation}
On the other hand, we observe that, for $p=1,
\cdots, n ,$
\begin{align}\label{11}
{} -2 \langle T_p u_i, \phi_i^p\rangle & = -2 \langle T_p u_i,B_p u_i \rangle +2 \langle T_p u_i, \sum_{j=1}^k a_{ij}^p u_j \rangle \nonumber\\
{} & =  2\langle u_i,T_p B_p u_i \rangle + 2 \sum_{j=1}^k \overline{a_{ij}^p} \langle T_p u_i, u_j \rangle  \nonumber\\
{} & = 2\langle u_i,T_p B_p u_i\rangle +2 \sum_{j=1}^k \overline{a_{ij}^p} c_{ij}^p,
\end{align}
where $c_{ij}^p=\langle T_pu_i,u_j \rangle$. \\
Note that, since $T_p$ is skew-symmetric, we have
$c_{ij}^p=- \overline{c_{ji}^p}$ for $1 \le p \le n$. \\
Therefore, using (\ref{b'}) and taking the real part of both sides of (\ref{11}), we
obtain, for any constant $\alpha_i>0$,
\begin{align}\label{g}
{}  2 Re\langle T_p B_p u_i, u_i \rangle + 2 \sum_{j=1}^k Re \left(\overline{a_{ij}^p} c_{ij}^p\right) & =  -2 Re \langle \phi_i^p, T_p u_i \rangle \nonumber\\
{}  & = 2 Re \langle \phi_i^p, -T_pu_i+ \sum_{j=1}^k c_{ij}^p u_j \rangle \nonumber \\
{}  & \leq  \alpha_i \|\phi_i^p\|^2 + \frac{1}{\alpha_i} \|-T_pu_i+\sum_{j=1}^k c_{ij}^p u_j\|^2 \nonumber \\
{} & =  \alpha_i \|\phi_i^p\|^2 +
\frac{1}{\alpha_i}\Big(\|T_pu_i\|^2-\sum_{j=1}^k\left|c_{ij}^p\right|^2\Big).
\end{align}
Multiplying (\ref{g}) by $f(\lambda_i)$ and taking $\displaystyle \alpha_i=\frac{\alpha(\lambda_{k+1}-\lambda_i)g(\lambda_i)}{f(\lambda_i)}$, where $\alpha$ is a positive constant and $i \le k$, we infer from
(\ref{e})
\begin{align}\label{h}
{}& 2f(\lambda_{i})  \bigg(Re \langle T_pB_p
u_i,u_i\rangle   +  \sum_{j=1}^k Re \left(\overline{a_{ij}^p} c_{ij}^p \right)\bigg) \nonumber\\
{}& \leq \alpha_if(\lambda_i)\|\phi_i^p\|^2 + \frac{1}{\alpha_i}
f(\lambda_i)\bigg(\|T_pu_i\|^2-\sum_{j=1}^k\left|c_{ij}^p\right|^2\bigg)\nonumber \\
{}&=\alpha (\lambda_{k+1}-\lambda_i)g(\lambda_i)\|\phi_i^p\|^2+\frac{1}{\alpha}\frac{\big(f(\lambda_i)\big)^2}{(\lambda_{k+1}-\lambda_i)g(\lambda_i)}\bigg(\|T_pu_i\|^2-\sum_{j=1}^k\left|c_{ij}^p\right|^2\bigg)\nonumber\\
{} & \leq \alpha g(\lambda_i) \langle [A,B_p] u_i, B_p u_i\rangle + \alpha g(\lambda_i) \lambda_{i}\|B_p u_i\|^2-\alpha g(\lambda_i) \sum_{j=1}^k
\lambda_j \left|a_{ij}^p\right|^2- \alpha g(\lambda_i)\lambda_i\|\phi_i^p\|^2 \nonumber \\
{}&+\frac{1}{\alpha}\frac{\big(f(\lambda_i)\big)^2}{(\lambda_{k+1}-\lambda_i)g(\lambda_i)}\bigg(\|T_p u_i\|^2-\sum_{j=1}^k {\left|c_{ij}^p\right|}^2 \bigg).
\end{align}
Summing over $i=1,\cdots,k$ and using (\ref{c}), we get
\begin{align}\label{k}
{} & 2\sum_{i=1}^k f(\lambda_i) Re \langle T_p B_p u_i,u_i\rangle+2\sum_{i,j=1}^k f(\lambda_i)Re\left(\overline{a_{ij}^p} c_{ij}^p\right)  \nonumber\\
{} & \leq \alpha \sum_{i=1}^kg(\lambda_i) \langle [A,B_p] u_i,B_p u_i \rangle + \alpha \sum_{i=1}^k \lambda_ig(\lambda_i) \|B_pu_i\|^2 \nonumber\\
{} &  -\alpha \sum _{i,j=1}^k \lambda_j g(\lambda_i) |a_{ij}^p|^2 - \alpha \sum_{i=1}^k \lambda_i g(\lambda_i)  \|\phi_i^p\|^2 \nonumber \\
{} & +\frac{1}{\alpha}\sum_{i=1}^k\frac{\big(f(\lambda_i)\big)^2}{(\lambda_{k+1}-\lambda_i)g(\lambda_i)}\bigg(\|T_pu_i\|^2-\sum_{j=1}^k
|c_{ij}^p|^2 \bigg)\nonumber\\
{} & = \alpha \sum_{i=1}^kg(\lambda_i) \langle [A,B_p] u_i,B_p u_i \rangle +\alpha \sum _{i,j=1}^k (\lambda_{i}-\lambda_j)g(\lambda_i) |a_{ij}^p|^2 \nonumber\\
{} & +\frac{1}{\alpha}\sum_{i=1}^k\frac{\big(f(\lambda_i)\big)^2}{(\lambda_{k+1}-\lambda_i)g(\lambda_i)}\bigg(\|T_pu_i\|^2-\sum_{j=1}^k|c_{ij}^p|^2 \bigg).
\end{align}
Since $a_{ij}^p=\overline{a_{ji}^p}$ and
$c_{ij}^p=-\overline{c_{ji}^p}$, we have 
\begin{align}\label{l}
2 Re \left(\sum_{i,j=1}^kf(\lambda_i)\overline{a_{ij}^p}c_{ij}^p\right) =Re\Big(\sum_{i,j=1}^k(f(\lambda_i)-f(\lambda_j))\overline{a_{ij}^p}c_{ij}^p\Big)
\end{align}
Using that $|c_{ij}^p|^{2}=|c_{ji}^p|^{2}$, we find
\begin{equation}\label{m}
-\frac{1}{\alpha}\sum_{i,j=1}^k\frac{\big(f(\lambda_i)\big)^2}{g(\lambda_i)(\lambda_{k+1}-\lambda_i)}|c_{ij}^p|^2=\frac{-1}{2\alpha}\sum_{i,j=1}^k\Big[\frac{\big(f(\lambda_i)\big)^2}{g(\lambda_i)(\lambda_{k+1}-\lambda_i)}+\frac{\big(f(\lambda_j)\big)^2}{g(\lambda_j)(\lambda_{k+1}-\lambda_j)}\Big]|c_{ij}^p|^2.
\end{equation}
Moreover, 
\begin{align}\label{n}
\displaystyle \alpha \sum_{i,j=1}^k g(\lambda_i)(\lambda_i-\lambda_j)|a_{ij}^p|^2
 =\displaystyle \frac{\alpha}{2} \sum_{i,j=1}^k \Big(g(\lambda_i)-g(\lambda_j)\Big)(\lambda_i-\lambda_j)|a_{ij}^p|^2. 
\end{align}
Thus we infer from (\ref{k}),(\ref{l}),(\ref{m}) and
(\ref{n})
\begin{align}\label{o}
{} & 2 \sum_{i=1}^k f(\lambda_i) Re \langle T_pB_p u_i, u_i \rangle +\sum_{i,j=1}^k\Big(f(\lambda_i)-f(\lambda_j)\Big)Re\left(\overline{a_{ij}^p} c_{ij}^p \right)\nonumber \\
{} & \leq  \alpha\sum_{i=1}^k g(\lambda_i)\langle [A,B_p] u_i, B_p u_i\rangle +\frac{1}{\alpha} \sum_{i=1}^k \frac{\big(f(\lambda_i)\big)^2}{g(\lambda_i)(\lambda_{k+1}-\lambda_i)}\|T_p u_i\|^2 \nonumber\\
{} & +\frac{\alpha}{2} \sum_{i,j=1}^k \Big(g(\lambda_i)-g(\lambda_j)\Big)(\lambda_i-\lambda_j)|a_{ij}^p|^2 \nonumber\\ {}&-\frac{1}{2\alpha}\sum_{i,j=1}^k
\Big[\frac{\big(f(\lambda_i)\big)^2}{g(\lambda_i)(\lambda_{k+1}-\lambda_i)}+\frac{\big(f(\lambda_j)\big)^2}{g(\lambda_j)(\lambda_{k+1}-\lambda_j)}\Big]|c_{ij}^p|^2.
\end{align}
But
\begin{align}
 \sum_{i,j=1}^k\Big(f(\lambda_j)-f(\lambda_i)\Big)&Re\left(\overline{a_{ij}^p}c_{ij}\right) \leq  \frac{\alpha}{2}\sum_{i,j=1}^k \Big(g(\lambda_j)-g(\lambda_i)\Big)(\lambda_i-\lambda_j)|a_{ij}^p|^2\nonumber\\
&+\frac{1}{2\alpha}\sum_{i,j=1}^k\frac{\Big(f(\lambda_j)-f(\lambda_i)\Big)^2}{(\lambda_i-\lambda_j)^2}\frac{\lambda_i-\lambda_j}{g(\lambda_j)-g(\lambda_i)}|c_{ij}^p|^2.
\end{align}
From the condition (\ref{cond}) satisfied by $f$ and $g$, we infer
\begin{align}\label{p}
{}  \displaystyle \sum_{i=1}^k \Big(f(\lambda_j)-&f(\lambda_i)\Big)Re\left(\overline{a_{ij}^p}c_{ij}^p\right) \leq  \displaystyle \frac{\alpha}{2} \sum_{i,j=1}^k \Big(g(\lambda_j)-g(\lambda_i)\Big)(\lambda_i-\lambda_j)|a_{ij}^p|^2 \nonumber \\
{}  & \displaystyle
+\frac{1}{2\alpha}\sum_{i,j=1}^k\bigg(\frac{(f(\lambda_i))^2}{g(\lambda_i)(\lambda_{k+1}-\lambda_i)}+\frac{(f(\lambda_j))^2}{g(\lambda_j)(\lambda_{k+1}-\lambda_j)}\bigg)|c_{ij}^p|^2.
\end{align}
Hence, taking sum on $p$, from 1 to $n$, in (\ref{o}), we find
\begin{align}\label{q}
{}  2\sum_{i=1}^k &\sum_{p=1}^n f(\lambda_i) Re\langle T_pB_p u_i, u_i \rangle \nonumber \\
{}  \leq \alpha \sum_{i=1}^k
\sum_{p=1}^ng(\lambda_i)& \langle [A,B_p] u_i,
B_p u_i
\rangle+\frac{1}{\alpha}\sum_{i=1}^k\sum_{p=1}^n\frac{\big(f(\lambda_i)\big)^2}{(\lambda_{k+1}-\lambda_i)g(\lambda_i)}\|T_p u_i\|^2.
\end{align}
Since $B_p$ is symmetric and $T_p$ is skew-symmetric, we have for all $p \le n$,
\begin{align*}
2 Re \langle T_pB_p u_i,u_i\rangle & =\langle T_pB_p u_i,u_i\rangle+\overline{\langle T_pB_p u_i,u_i\rangle}\\
& =\langle T_pB_p u_i,u_i\rangle-\overline{\langle u_{i},B_p T_p u_i\rangle}\\
& =\langle [T_p,B_p]u_i,u_i\rangle
\end{align*}
and inequality (\ref{q}) becomes
\begin{align}\label{r}
{}  \displaystyle
\sum_{i=1}^k&\sum_{p=1}^nf(\lambda_i)\langle
[T_p,B_p]u_i,u_i \rangle \nonumber \\
{} \displaystyle \leq \alpha \sum_{i=1}^k \sum_{p=1}^n
g(\lambda_i) &\langle[A,B_p]u_i,B_p
u_i\rangle + \frac{1}{\alpha}\sum_{i=1}^k \sum_{p=1}^n\frac{\big(f(\lambda_i)\big)^2}{(\lambda_{k+1}-\lambda_i)g(\lambda_i)}
\|T_pu_i\|^2,
\end{align}
or equivalently
\begin{align}\label{s}
{} \displaystyle \alpha^2 \sum_{i=1}^k
\sum_{p=1}^ng(\lambda_i) &\langle
[A,B_p]u_i,B_p u_i \rangle \nonumber \\
{}  -\alpha \sum_{i=1}^k\sum_{p=1}^n
f(\lambda_i)\langle [T_p,B_p]u_i,u_i \rangle +&
\sum_{i=1}^k
\sum_{p=1}^n\frac{\big(f(\lambda_i)\big)^2}{(\lambda_{k+1}-\lambda_i)g(\lambda_i)}\|T_pu_i\|^2 \geq
0.
\end{align}
To prove inequality (\ref{a}), it suffices to show that
\begin{equation}\label{positive}
\sum_{i=1}^k\sum_{p=1}^ng(\lambda_i)\langle[A,B_p]u_i,
B_p u_i\rangle \ge 0 .
\end{equation}
In fact, if this is the case, the discriminant of the quadratic
polynomial (\ref{s}) must be nonpositive, i.e.
\begin{align}\label{disc}
{}  \bigg(\sum_{i=1}^k \sum_{p=1}^n&f(\lambda_i)\langle[T_p,B_p]u_i,u_i \rangle \bigg)^2\nonumber \\
{} -4\bigg(\sum_{i=1}^k
\sum_{p=1}^ng(\lambda_i)&\langle[A,B_p]u_i,B_p
u_i\rangle\bigg)\bigg(\sum_{i=1}^k\sum_{p=1}^n\frac{\Big(f(\lambda_i)\Big)^2}{(\lambda_{k+1}-\lambda_i)g(\lambda_i)}\|T_pu_i\|^2\bigg)\leq
0.
\end{align}
which yields the theorem. We note that if we replace $T_p$ by $-T_p$, inequality (\ref{s}) holds. Thus we can deduce that it holds for all real $\alpha$ and not only $\alpha >0$ proving that the coefficient of the quadratic term, i.e. $\sum_{i=1}^k\sum_{p=1}^ng(\lambda_i)\langle[A,B_p]u_i,
B_p u_i\rangle$,  is nonnegative. If it is equal to zero, then $\sum_{i=1}^k \sum_{p=1}^nf(\lambda_i)\langle[T_p,B_p]u_i,u_i \rangle$ is also equal to 0 and the theorem trivially holds.
\end{proof}
\begin{remark}
\begin{itemize}
\item In the definition of $\Im_{\lambda}$, the functions $f$ and $g$ can be defined only on a discrete set of eigenvalues. 
\item One can formulate Theorem \ref{first theorem} as in \cite{HarlStub} for $z \in ]\lambda_{k},\lambda_{k+1}]$ (it suffices to replace, in the hypothesis and in the inequality, $\lambda_{k+1}$ by $z$). 
 \item The result of Theorem \ref{first theorem} can also be stated, as in \cite {HarlStub} or \cite{HarlStub2}, in the general situation where the spectrum of $A$ is not purely discrete and its point spectrum is nonempty.
\item Taking $f=g=1$ in (\ref{a}), we obtain inequality (\ref{Ashb-Her}) of Ashbaugh and Hermi.
\end{itemize}
\end{remark}
If the operators $T_{p}$ are chosen such that
$T_{p}=\left[A,B_{p}\right]$, then
$\left[T_{p},B_{p}\right]=\left[\left[A,B_{p}\right],B_{p}\right]$.
Applying Theorem \ref{first theorem} in this context and using the
obvious identity $\langle
\left[A,B_{p}\right]u_{i},B_{p}u_{i}\rangle=-\frac{1}{2} \langle
\left[\left[A,B_{p}\right],B_{p}\right]u_{i},u_{i}\rangle$, we
obtain
\begin{corollary}\label{first corollary}
 Let $A$ : $\mathcal{D}\subset \mathcal{H}\longrightarrow
\mathcal{H}$ be a self-adjoint operator defined on a dense domain
$\mathcal{D}$, which is semibounded below and has a discrete
spectrum $\lambda_1 \leq \lambda_2\leq \lambda_3...$. 
Let $\{B_p:A(\mathcal{D})\longrightarrow \mathcal{H}\}_{p=1}^n$ be a
collection of symmetric operators, leaving $\mathcal{D}$ invariant. We denote by $\left\{u_i\right\}_{i=1}^{\infty}$ a basis of orthonormal
eigenvectors of $A$, $u_{i}$ corresponding to $\lambda_i$. If for $k\ge 1$ we have $\lambda_{k+1}>\lambda_k$,  
then for any $(f,g) \in \Im_k $, 
\begin{align} \label {algebraic ineq}
\bigg[\sum_{i=1}^k \sum_{p=1}^n &f(\lambda_i)\langle[A,
B_p]u_i,B_{p}u_i\rangle\bigg]^2 \nonumber \\
{}\leq \bigg[\sum_{i=1}^k  \sum_{p=1}^n
g(\lambda_i)&\langle [A,B_p]u_i, B_p
u_i\rangle \bigg] \bigg[\sum_{i=1}^k \sum_{p=1}^n
\frac{\Big(f(\lambda_i)\Big)^2}{g(\lambda_i)(\lambda_{k+1}-\lambda_i)} \|\left[A,B_{p}\right]
u_i\|^2\bigg].
\end{align}
\end{corollary}
\begin{remark}
\begin{itemize}
\item As for Theorem \ref{first theorem}, Corollary \ref{first corollary} can be stated in the general case where the spectrum of $A$ is not totally discrete. 
\item For $f(x)=g(x)=(\lambda_{k+1}-x)^{2}$, inequality (\ref{algebraic ineq}) becomes the abstract inequality which gives the Yang type inequalities for Laplacians and Schr\"odinger operators (see \cite{AshbHer1}, \cite{Soufi.Harl.Ilias}, \cite{HarlStub} and \cite{LevPar}).
\item For $f(x)=g(x)=(\lambda_{k+1}-x)^{\alpha}, {\rm with}\; \alpha \le 2$, we recover a Harrell and Stubbe inequality (\cite {HarlStub}, \cite{AshbHer3})).
\item We can easily deduce from the inequality (\ref {algebraic ineq}) new universal inequalities in many different geometric situations (Dirichlet Laplacian on domains of Submanifolds of Euclidean (or symmetric) spaces as in \cite{Soufi.Harl.Ilias}, Hodge de Rham Laplacian or the square of a Dirac operator,and more generally a Laplacian acting on sections of a Riemannian vector bundle on a submanifold of a Euclidean (or symmetric) space). 
\end{itemize}
\end{remark}
\section {\textbf{Application to the polyharmonic operators}}
In this section, using  Theorem \ref{first theorem}, we will show how to derive universal inequalities for the eigenvalues of a polyharmonic problem.  For a power of the Laplacian and with a particular choice of $f$ and $g$, one can derive inequality (\ref{WuCao}) and inequality (\ref {8'}). \\
In fact, throughout this section we
assume that $A=Q^l$, such that $Q$ is a symmetric self-adjoint
operator given by \begin{equation*} \displaystyle Q=-\sum_{p=1}^n
T_p^2, \end{equation*} where $T_p$ are skew-symmetric operators for
$p=1,...,n$, with $[Q,T_p]=0$ and $[T_m,B_p]=\delta_{mp}$. \\
First we need
to calculate the following expressions
\begin{equation}\label{exp1}
 \sum_{p=1}^n \langle [A,B_p]u_i,B_p u_i \rangle
\end{equation}
and
\begin{equation}\label{exp2}
 \sum_{p=1}^n \langle T_p u_i, T_p u_i \rangle= \langle Q u_i,u_i \rangle.
\end{equation}
For this purpose the following two results of Hook (see proposition 3 in \cite{Hook1} and Theorem 1 in \cite{Hook2}) will be useful. \\
The first one is
\begin{lemma}\label{Hook1}
Under the circumstances stated above, we have
\begin{equation*}
[A,B_p]=[Q^l,B_p]=-2lQ^{l-1}T_p
\end{equation*}
and
\begin{equation*}
\sum_{p=1}^n[B_p,[A,B_p]]=2l(2l+n-2)Q^{l-1}.
\end{equation*}
\end{lemma}
And the second one is the following
\begin{theorem}\label{Hook2}
Let $V$ be a real or complex inner product space with inner product
$\langle.,.\rangle$. Let $D$ be a linear submanifold of $V$ and let
$Q:D\longrightarrow V$ be a linear operator in $V$. Suppose $l$ is a
positive integer and $u$ is a fixed vector such that for all $0 \leq
r \leq q \leq l$,
\begin{equation*}
 \arrowvert \langle Q^q u, u \rangle \arrowvert = \arrowvert \langle Q^{q-r} u, Q^r u \rangle \arrowvert.
\end{equation*}
Then, for all integers $0 \leq r \leq q \leq l$, when $q$ is even,
we have
\begin{equation}\label {hok}
 \arrowvert \langle Q^r u,u \rangle \arrowvert \leq \arrowvert \langle Q^qu,u \rangle \arrowvert^{r/q} \langle u,u \rangle^{1-r/q}.
\end{equation}
This inequality is satisfied for $q$ odd and $0 \leq r \leq q \leq
l$, if in addition to the above, there is a family of operators
$\{T_p\}_{p=1}^n$ such that
\begin{equation*}
 \arrowvert \langle Q^q u,u \rangle \arrowvert=\arrowvert \sum_{p=1}^n \langle T_p Q^{q-r}u,T_p Q^{r-1}u \rangle\arrowvert
\end{equation*}
holds for all $0 \leq r \leq q \leq l$.
\end{theorem}
Applying Lemma \ref{Hook1}, we obtain
\begin{align*}
 \sum_{p=1}^n \langle [A,B_p]u_i, B_p u_i \rangle & =\frac{1}{2}\sum_{p=1}^n \langle [B_p,[A,B_p]]u_i,u_i\rangle \\
 & = l(2l+n-2) \langle Q^{l-1}u_i,u_i \rangle.
\end{align*}
Therefore, if $l$ is odd, then we have
\begin{equation*}
 \sum_{p=1}^n \langle [A,B_p]u_i,B_p u_i \rangle=l(2l+n-2)\|Q^{\frac{l-1}{2}}u_i\|^2
\end{equation*}
and if $l$ is even, then
\begin{equation*}
 \sum_{p=1}^n \langle [A,B_p]u_i,B_p u_i \rangle = l(2l+n-2)\sum_{p=1}^n\|T_p Q^{\frac{l-2}{2}}u_i\|^2
\end{equation*}
The conditions of Theorem \ref{Hook2} are satisfied by our
operator $Q$. So inequality (\ref{hok}) is valid for all $0 \le r \le q \le l$ without parity condition on $q$. Applying this inequality (\ref{hok}) with $r=l-1$ and $q=l$, we obtain
\begin{align} \label{5}
\displaystyle \sum_{p=1}^n \langle [A,B_p]u_i,B_p u_i\rangle & = l(2l+n-2) \langle Q^{l-1}u_i,u_i \rangle \nonumber \\
 &\le l(2l+n-2)\langle Q^l u_i,u_i \rangle ^{\frac{l-1}{l}} \langle u_i,u_i \rangle ^{1-\frac{l-1}{l}}\nonumber\\
 &= l(2l+n-2)\lambda_i^{\frac{l-1}{l}} 
\end{align}
and with $r=1$ and $q=l$, we obtain
\begin{align}\label{6}
\displaystyle \sum_{p=1}^n \|T_p u_i\|^2 =\langle Q u_{i},u_{i} \rangle \le \langle Q^{l} u_{i},u_{i} \rangle^{\frac{1}{l}}\langle u_{i},u_{i} \rangle ^{1-\frac{1}{l}} \leq \lambda_i^{\frac{1}{l}}.
\end{align}
Since $[T_p,B_p]=1$, one gets
\begin{equation}\label{8}
 \langle [T_p,B_p]u_i,u_i \rangle=1.
\end{equation}
Then using inequalities (\ref{5}), (\ref{6}) and (\ref{8}) together  with inequality (\ref{a}), we obtain
\begin{align*}
 {} n^2 \bigg[\sum_{i=1}^k f(\lambda_i)\bigg]^2  \le 
4l(2l+n-2) \bigg(\sum_{i=1}^k g(\lambda_i)\lambda_i^{\frac{l-1}{l}}\bigg)\bigg(\sum_{i=1}^k
\frac{\Big(f(\lambda_i)\Big)^2}{g(\lambda_i)(\lambda_{k+1}-\lambda_i)}\lambda_i^{\frac{1}{l}}\bigg)
\end{align*} 
or equivalently
\begin{align} \label{poly}
{}\sum_{i=1}^k f(\lambda_i) \leq
 \frac{2}{n} \sqrt{l(2l+n-2)} \bigg(\sum_{i=1}^k g(\lambda_i)\lambda_i ^{\frac{l-1}{l}}\bigg)^{\frac{1}{2}} \bigg(\displaystyle
\sum_{i=1}^k\frac{\Big(f(\lambda_i)\Big)^2}{g(\lambda_i)(\lambda_{k+1}-\lambda_i)}\lambda_i^{\frac{1}{l}}\bigg)^{\frac{1}{2}}.
\end{align}
Now the operators $A=(-\Delta)^l$, $Q=-\Delta$, $B_p=x_p$,
$p=1,\ldots,n$, where $x_1,\ldots,x_n$ are Euclidean coordinates,
and $T_p=\frac{\partial}{\partial x_p}$ fit the setup of this section. Thus, taking $f(x)=\Big(g(x)\Big)^2=(\lambda_{k+1}-x)$,  we can obtain inequality
(\ref{WuCao}) of Wu and Cao.
\begin{remark}
For the special case $l=2$ (i.e the clamped plate problem) and the same values of $f$ and $g$ as above, we obtain inequality (\ref{4}) of Cheng and Yang. We observe that this inequality can also be obtained easily by a simple calculation from our inequality (\ref {a}). In fact, taking $A=\Delta^2$, $B_p=x_p$, $p=1,\ldots,n$ and $T_p=\frac{\partial}{\partial x_p}$, we first observe that $[T_p,B_p]=1$. Hence, we have
\begin{equation}\label{c1}
\sum_{p=1}^n \langle [T_p,B_p]u_i,u_i \rangle=n,
\end{equation}
moreover
\begin{equation*}
[A,B_p]u_i=[\Delta^2,x_p]u_i=4\frac{\partial}{\partial x_p}\Delta
u_i.
\end{equation*}
Then
\begin{align*}
{} [B_p,[A,B_p]]u_i & =4[x_p,\frac{\partial}{\partial x_p}\Delta]u_i\\
{} & = -4 \Big(\Delta + 2\big( \frac{\partial}{\partial x_p} \big)^2
\Big)u_i.
\end{align*}
It follows that
\begin{align}
{} \sum_{p=1}^n \langle [A,B_p]u_i,B_p u_i \rangle
&=\frac{1}{2}\sum_{p=1}^n \langle [B_p,[A,B_p]]u_i,u_i\rangle
\nonumber\\
{} & = 2(n+2) \langle -\Delta u_i,u_i \rangle \nonumber\\
{} & \leq 2(n+2)\Big(\|\Delta u_i\|^2\|u_i\|^2 \Big)^{\frac{1}{2}}\label{d1}\\
{}\label{e1} & = 2(n+2) \lambda_i^{\frac{1}{2}}.
\end{align}
Now
\begin{equation}\label{f1}
\sum_{p=1}^n \langle T_p u_i, T_p u_i \rangle=\langle -\Delta u_i,u_i \rangle \leq \lambda_i^{\frac{1}{2}},
\end{equation}
where we used the Cauchy-Schwarz inequality to derive (\ref{d1}) and (\ref{f1}).\\
Substituting (\ref{c1}), (\ref{e1}) and (\ref{f1}) into (\ref{a}) and taking $f(x)=\Big(g(x)\Big)^2=(\lambda_{k+1}-x)$, we obtain inequality (\ref{4}).
\end{remark}
On the other hand, if we take $f(x)=g(x)=(\lambda_{k+1}-x)^2$, in (\ref{poly}),
 we get the following inequality obtained in \cite{Cheng-Ichik-Mamet} (see inequality (2.27) therein) 
\begin{align}\label{14}
\bigg[\sum_{i=1}^k&(\lambda_{k+1}-\lambda_i)^2\bigg]^2\nonumber\\
\leq \frac{4l(2l+n-2)}{n^2}\bigg(\sum_{i=1}^k&(\lambda_{k+1}-\lambda_i)^2\lambda_i^{\frac{l-1}{l}}\bigg) \bigg(\sum_{i=1}^k(\lambda_{k+1}-\lambda_i)\lambda_i^{\frac{1}{l}}\bigg).
\end{align}
Using the following variant of Chebyshev inequality (see
Lemma 1 in \cite{Cheng-Ichik-Mamet}), one can deduce a generalized Yang inequality
\begin{lemma}\label{lemma3}
 Let $A_i$, $B_i$ and $C_i$, $i=1,\ldots,k$, verify $A_1\ge A_2\ge \ldots \ge A_k \ge 0$, $0\leq B_1\leq B_2\leq \ldots \leq B_k$ and $0\leq C_1\leq C_2\leq \ldots \leq C_k$, respectively. Then, we have
\begin{equation*}
\sum_{i=1}^k A_i^2 B_i \sum_{i=1}^k A_i C_i \leq \sum_{i=1}^k A_i^2
\sum_{i=1}^k A_i B_i C_i.
\end{equation*}
\end{lemma}
In fact if we apply this Lemma to the right side of inequality (\ref
{14}), with $A_i=\lambda_{k+1}-\lambda_i$,
$B_i=\lambda_i^{\frac{l-1}{l}}$ and $C_i=\lambda_i^{\frac{1}{l}}$,
we obtain,
\begin{equation}\label{15}
\sum_{i=1}^k(\lambda_{k+1}-\lambda_i)^2\leq \frac{4l(n+2l-2)}{n^2}\sum_{i=1}^k(\lambda_{k+1}-\lambda_i)\lambda_i.
\end{equation}
which, in the case where $A=(-\Delta)^l$, $Q=-\Delta$, $B_p=x_p$,
$p=1,\ldots,n$ and $T_p=\frac{\partial}{\partial x_p}$, gives us
inequality (\ref {8'}) of Cheng, Ichikawa and Mametsuka (see
inequality (1.11) in \cite{Cheng-Ichik-Mamet}).\\
Finally, we note that considering other choices of values for the couple $(f,g)$ lead to many new inequalities.
\section{Applications to the Kohn Laplacian on the Heisenberg group}
In this section, we consider the $2n+1$-dimensional Heisenberg group $\mathbb{H}^n$, which is the space $\R^{2n+1}$ equipped with the non-commutative group law
\begin{equation*}
 (x,y,t)(x',y',t')=\bigg(x+x',y+y',t+t'+\frac{1}{2}\bigg)(\langle x',y \rangle_{\R^n}-\langle x,y' \rangle_{\R^n}),
\end{equation*}
where $x,\, x',\,y,\, y'\,\in \,\R^n$, $t$ and $t' \,\in\,\R$. We
denote by $\mathcal{H}^n$ its Lie algebra, it has a basis formed by
the following vector fields $T=\frac{\partial}{\partial t}$,
$X_p=\frac{\partial}{\partial
x_p}+\frac{y_p}{2}\frac{\partial}{\partial t}$ and
$Y_p=\frac{\partial}{\partial
y_p}-\frac{x_p}{2}\frac{\partial}{\partial t}$. We note that the
only non-trivial commutators are $[Y_p,X_q]=T \delta_{pq}$. Let
$\Delta_{\mathbb{H}^n}$ denote the real Kohn-Laplacian in the
Heisenberg group $\mathbb{H}^n$. It is given by
\begin{align*}
 \Delta_{\mathbb{H}^n} & =\sum_{p=1}^n X_p^2+Y_p^2\\
&= \Delta^{\mathbb{R}^{2n}}_{xy}+\frac{1}{4}(|x|^{2}+|y|^{2})\frac{\partial^{2}}{{\partial t}^{2}}+ \frac{\partial}{\partial t}\sum_{p=1}^{n}\left(y_{p}\frac{\partial}{\partial x_{p}}-x_{p}\frac{\partial}{\partial y_{p}}\right).\end{align*}
We are concerned here with the following eigenvalue problem:
\begin{equation}\label{Kohn Lap}
\begin{cases}
(-\Delta_{\mathbb{H}^n})^l u  = \lambda u \,\,\,\,\text{in}\,\,\Omega,\\
u=\displaystyle{\frac{\partial u}{\partial \nu}=\ldots=\frac{\partial^{l-1} u}{\partial \nu^{l-1}}}=0\,\,\,\,\text{on}\,\,\partial\Omega,
\end{cases}
\end{equation}
where $\Omega$ is a bounded domain in $\mathbb{H}^n$, with smooth boundary $\partial \Omega$, $\nu$ is the unit outward normal to $\partial \Omega$ and $l\ge 1$ is any positive integer. We denote by $L=-\Delta_{\mathbb{H}^n}$ and $\nabla_{\mathbb{H}^n}=(X_1,\ldots,X_n,Y_1,\ldots,Y_n)$.\\
We let
\begin{equation*}
 0<\lambda_1 \leq \lambda_2 \leq \ldots \leq \lambda_k \leq \ldots \rightarrow + \infty
\end{equation*}
denote the eigenvalues of problem (\ref{Kohn Lap}) with
corresponding eigenfunctions $u_1,\, u_2,\ldots, u_k,\ldots$ in
$S_{0}^{l,2}(\Omega)$. Here $S^{l,2}(\Omega)$ is the Hilbert space
of the functions $u$ in $L^2(\Omega)$ such that $X_p u,\, Y_p u,\,
X^2_{p} u, Y^2_p u,\ldots,\,X_p^l(u)$, $Y_p^l(u)\,\in\,L^2(\Omega)$,
and $S_0^{l,2}$ denotes the closure of
$\textit{C}_{0}^{\infty}(\Omega)$ with respect to the Sobolev norm
\begin{equation*}
 \|u\|_{S^{l,2}}^2=\displaystyle \int_{\Omega}\bigg( \sum_{d=1}^l \Big(\sum_{p=1}^n \arrowvert X_p^d u\arrowvert^2+\sum_{p=1}^n\arrowvert Y_p^d u\arrowvert^2\Big) +\arrowvert u \arrowvert^2\bigg) dx dy dt.
\end{equation*}
We orthonormalize the eigenfunctions $u_i$ so that; $\forall \,i,\,j\,\geq 1$,
\begin{equation*}
\langle u_i,u_j\rangle_{L^2}=\int_{\Omega}u_i u_j dx dy dt=\delta_{ij}.
\end{equation*}

In all this paragraph, our results can be stated in a general form using functions $f$ and $g \in \Im_{\lambda_{k+1}}$ as in the first part of this paper, but we limit ourselves to the case $f(x)=g(x)=(\lambda_{k+1}-x)^2$. This gives us new bounds of the Yang type for eigenvalues of problem (\ref{Kohn Lap}) which improve earlier ones obtained by Niu and Zhang \cite{Niu-Zhang}.\\
We also note that we
must treat the three following cases independently: the case when
$l=1$, the case when $l=2$ and the case when $l\ge 3$. This is
essentially due to the difference of the calculations in these three cases.
\subsection{The case when $\textit{\textbf{l}}$ = 1}
In this subsection, we are concerned with the case where $l=1$. The
result we obtain is a result proved earlier by the first author, El
Soufi and Harrell in \cite{Soufi.Harl.Ilias} and for which we give
here a different proof, more easily adapted to the other cases $l=2$
and $l\ge3$.
\begin{theorem} For any $k\ge 1$
\begin{equation}\label{m2}
\sum_{i=1}^k(\lambda_{k+1}-\lambda_i)^2\leq \frac{2}{n}\sum_{i=1}^k(\lambda_{k+1}-\lambda_i)\lambda_i.
\end{equation}
\end{theorem}
\begin{proof}
We will prove this theorem by applying inequality (\ref{a}) with
$A=L=-\Delta_{\mathbb{H}^n}$,
$B_1=x_1,\ldots,B_n=x_n,B_{n+1}=y_1,\ldots,B_{2n}=y_n$,
$T_1=X_1,\ldots,T_n=X_n,T_{n+1}=Y_1,\ldots,T_{2n}=Y_n$ and $f(x)=g(x)=(\lambda_{k+1}-x)^2$, namely,
\begin{align}\label{n2}
{}\bigg[\sum_{p=1}^n\sum_{i=1}^k(\lambda_{k+1}-\lambda_i)^2
\Big(&\langle[X_p,x_p]u_i, u_i\rangle_{L^2}+\langle [Y_p,y_p]u_i,u_i\rangle_{L^2}\Big)\bigg]^2\nonumber\\
\leq4\bigg[\sum_{p=1}^n\sum_{i=1}^k(\lambda_{k+1}-\lambda_i)^2\Big(&\langle
[L,x_p]u_i,x_p u_i\rangle_{L^2}+\langle [L,y_p]u_i,y_p
u_i\rangle_{L^2}\Big)\bigg]\times\nonumber\\
\bigg[\sum_{p=1}^n\sum_{i=1}^k&(\lambda_{k+1}-\lambda_i)\Big(\|X_pu_i\|_{L^2}^2+\|Y_pu_i\|_{L^2}^2\Big)\bigg].
\end{align}
By a straightforward calculation, we obtain $[L,x_p]u_i=-2X_pu_i$
and $[L,y_p]u_i=-2Y_pu_i$.\\
 Hence
\begin{align*}
 \langle[L,x_p]u_i,x_p u_i\rangle_{L^2}&=-2\int_{\Omega}X_pu_i.x_pu_i=2\int_{\Omega}u_i.X_p(x_pu_i) \nonumber\\
 &=2\int_{\Omega}u_i^2+2\int_{\Omega}x_pu_i.X_pu_i
\end{align*}
and
\begin{align*}
 \langle[L,y_p]u_i,y_p u_i\rangle_{L^2}&=-2\int_{\Omega}Y_pu_i.y_pu_i=2\int_{\Omega}u_i.Y_p(y_pu_i) \nonumber\\
 &=2\int_{\Omega}u_i^2+2\int_{\Omega}y_pu_i.Y_pu_i,
\end{align*} then
\begin{equation}\label{p2}
\langle[L,x_p]u_i,x_p u_i\rangle_{L^2}=\langle[L,y_p]u_i,y_p u_i\rangle_{L^2}=1.
\end{equation}
On the other hand, we have
\begin{equation}\label{r2}
[X_p,x_p]u_i=[Y_p,y_p]u_i=u_i
\end{equation}
and \begin{equation}\label{p1}\sum_{p=1}^n\|X_pu_i\|^2_{L^2}+
\sum_{p=1}^n\|Y_pu_i\|^2_{L^2}=\int_{\Omega}\arrowvert\nabla_{\mathbb{H}^n}u_i\arrowvert^2=\lambda_i.\end{equation}
Thus incorporating (\ref{p2}), (\ref{r2}) and (\ref{p1}) in
(\ref{n2}), we obtain (\ref{m2}).
\end{proof}
\begin{remark}
Inequality (\ref{m2}) improves the following inequality proved by
Niu and Zhang in \cite{Niu-Zhang} (see Remark 5.1 in
\cite{Soufi.Harl.Ilias})
\begin{equation*}
 \lambda_{k+1}-\lambda_k\leq \frac{2}{nk}\sum_{i=1}^k\lambda_i.
\end{equation*}
\end{remark}

\subsection{The case when $\textit{\textbf{l}}$ = 2}
In this subsection, we will derive the following
\begin{theorem}\label{5th theorem}
We have, for each $k=1,2,\ldots$,
\begin{align}\label{c3}
\sum_{i=1}^k(\lambda_{k+1}-\lambda_i)^2 \leq
\frac{2\sqrt{n+1}}{n}\bigg[\sum_{i=1}^k(\lambda_{k+1}-\lambda_i)\lambda_i^{\frac{1}{2}}\bigg]^{\frac{1}{2}}\bigg[\sum_{i=1}^k(\lambda_{k+1}-\lambda_i)^2\lambda_i^{\frac{1}{2}}\bigg]^{\frac{1}{2}}.
\end{align}
\end{theorem}
\begin{proof}
The key observation here is to apply Theorem \ref{first theorem}
with $A=L^2=(-\Delta_{\mathbb{H}^n})^2$, and as before
$B_1=x_1,B_2=x_2,\cdots,B_n=x_n,B_{n+1}=y_1,\cdots,B_{2n}=y_n$,
$T_1=X_1,\cdots,T_n=X_n,T_{n+1}=Y_{1},\cdots,T_{2n}=Y_{n}$ and $f(x)=g(x)=(\lambda_{k+1}-x)^2$. Thus we
have
\begin{align}\label{s2}
\bigg[\sum_{p=1}^n \sum_{i=1}^k
(\lambda_{k+1}-\lambda_i)^2&\Big(\langle[X_p,x_p]u_i,u_i\rangle_{L^2}+\langle
[Y_{p},y_p]u_i,u_i\rangle_{L^2}\Big)\bigg]^{2}\nonumber\\
{} \le 4 \bigg[ \sum_{p=1}^n \sum_{i=1}^k (\lambda_{k+1}-\lambda_i)^2 &\Big( \langle [L^2,x_p]u_i,x_pu_i\rangle_{L^2}+ \langle [L^2,y_p]u_i,y_pu_i\rangle_{L^2}\Big)\bigg]\nonumber\\
{} \times \bigg[\sum_{p=1}^n& \sum_{i=1}^k
(\lambda_{k+1}-\lambda_i)\Big( \|X_pu_i\|_{L^2}^2+\|Y_pu_i\|_{L^2}^2
\Big) \bigg]
\end{align}
but
\begin{align}\label{s'2}
\sum_{p=1}^n\|X_pu_i\|^2_{L^2}+\sum_{p=1}^n\|Y_pu_i\|^2_{L^2}& =\int_{\Omega}\arrowvert\nabla_{\mathbb{H}^n}u_i\arrowvert^2=\int_{\Omega}Lu_i.u_i \nonumber\\
{}&\le\bigg(\int_{\Omega}u_i^2\bigg)^{\frac{1}{2}}\bigg(\int_{\Omega}\Big(Lu_i\Big)^{2}\bigg)^{\frac{1}{2}}=\lambda_i^{\frac{1}{2}},\end{align}
thus
\begin{equation} \label{t2}
\sum_{p=1}^n \sum_{i=1}^k (\lambda_{k+1}-\lambda_i)\big(
\|X_pu_i\|_{L^2}^2+\|Y_pu_i\|_{L^2}^2 \big)=\sum_{i=1}^k
(\lambda_{k+1}-\lambda_i)\lambda_{i}^{\frac{1}{2}}.
\end{equation}
Using (\ref{r2}), we get
\begin{equation}\label{a3}
\langle[X_p,x_p]u_i,u_i\rangle_{L^2}=\langle[Y_p,y_p]u_i,u_i\rangle_{L^2}=1.
\end{equation}
Thus,
\begin{align}\label{u2}
\bigg[\sum_{p=1}^n \sum_{i=1}^k (\lambda_{k+1}&-\lambda_i)^2\left(\langle[X_p,x_p]u_i,u_i\rangle_{L^2}+\langle [Y_{p},y_p]u_i,u_i\rangle_{L^2}\right)\bigg]^{2}\\
&=4n^{2}\bigg[\sum_{i=1}^k
(\lambda_{k+1}-\lambda_i)^2\bigg]^{2}.\nonumber
\end{align}
On the other hand
\begin{align}\label{v2}
[L^2,x_p]u_i=L^2(x_p u_i)
-x_pL^2u_i&=L(x_{p}Lu_{i}-2X_{p}u_{i})-x_{p}L^{2}u_{i}\nonumber\\
&=-2X_pLu_i-2L(X_pu_i)
\end{align}
and the same identity holds with $y_p$ and $Y_p$.\\
We infer, using identities (\ref{r2}) and (\ref{v2})
\begin{align}\label{w2}
\langle[L^2,x_p]u_i,x_pu_i\rangle_{L^2}& =-2\int_{\Omega}X_pLu_i.x_pu_i-2\int_{\Omega}L(X_pu_i).x_pu_i\nonumber\\
= -2&\int_{\Omega}X_pLu_i.x_pu_i-2\int_{\Omega}X_pu_i.x_pLu_i+4\int_{\Omega}X_pu_i.X_pu_i\nonumber\\
=2&\int_{\Omega}Lu_i.X_p(x_pu_i)-2\int_{\Omega}x_pX_pu_i.Lu_i-4\int_{\Omega}X_p^2u_i.u_i\nonumber\\
=&2\int_{\Omega}Lu_i.u_i-4\int_{\Omega}X_p^2u_i.u_i.
\end{align}
Similarly, we have
\begin{align}\label{x2}
\langle[L^2,y_p]u_i,y_pu_i\rangle_{L^2}=
2\int_{\Omega}Lu_i.u_i-4\int_{\Omega}Y_p^2u_i.u_i.
\end{align}
Since
\begin{align*}
-\sum_{p=1}^n\int_{\Omega}X_p^2u_i.u_i-\sum_{p=1}^n\int_{\Omega}Y_p^2u_i.u_i &= \sum_{p=1}^n\|X_pu_i\|_{L^2}^2+\sum_{p=1}^n\|Y_pu_i\|_{L^2}^2\\
&=\int_{\Omega}Lu_i.u_i,\nonumber
\end{align*}
we have
\begin{align}\label{y2}
\bigg[ \sum_{p=1}^n \sum_{i=1}^k (\lambda_{k+1}&-\lambda_i)^2 \big( \langle [L^2,x_p]u_i,x_pu_i\rangle_{L^2}+ \langle [L^2,y_p]u_i,y_pu_i\rangle_{L^2}\big)\bigg] \nonumber\\
{}& =4(n+1)\sum_{i=1}^k (\lambda_{k+1}-\lambda_i)^2\int_{\Omega}Lu_i.u_i\nonumber\\
{}&\leq 4(n+1)\sum_{i=1}^k (\lambda_{k+1}-\lambda_i)^2\bigg(\int_{\Omega}u_i^2\bigg)^{\frac{1}{2}}\bigg(\int_{\Omega}\Big(Lu_i\Big)^{2}\bigg)^{\frac{1}{2}}\nonumber \\
{}& =4(n+1)\sum_{i=1}^k
(\lambda_{k+1}-\lambda_i)^2\lambda_i^{\frac{1}{2}}.
\end{align}
Incorporating (\ref{t2}), (\ref{u2}) and (\ref{y2}) in (\ref{s2}),
we get the result.
\end{proof}
We can easily obtain from inequality  (\ref{c3}) of Theorem \ref{5th theorem} an inequality of Yang-type.
\begin{corollary}
We have, for each $k\ge1$,
\begin{equation}\label{cor1}
\sum_{i=1}^k(\lambda_{k+1}-\lambda_i)^2\leq \frac{4(n+1)}{n^2}\sum_{i=1}^k(\lambda_{k+1}-\lambda_i)\lambda_i.
\end{equation}
\end{corollary}
\begin{proof}
 Inequality (\ref{c3}) is equivalent to
\begin{equation*}
\bigg[\sum_{i=1}^k(\lambda_{k+1}-\lambda_i)^2\bigg]^2\leq\frac{4(n+1)}{n^2}\bigg[\sum_{i=1}^k(\lambda_{k+1}-\lambda_i)\lambda_i^{\frac{1}{2}}\bigg]\bigg[\sum_{i=1}^k(\lambda_{k+1}-\lambda_i)^2\lambda_i^{\frac{1}{2}}\bigg].
\end{equation*}
Now applying Lemma \ref{lemma3} with $A_i=\lambda_{k+1}-\lambda_i$
and $B_i=C_i=\lambda_i^{\frac{1}{2}}$, we obtain inequality
(\ref{cor1}).
\end{proof}
\begin{remark}
 Inequality (\ref{c3}) is sharper than the following one found by Niu and Zhang \cite{Niu-Zhang}
\begin{equation*}
 \lambda_{k+1}-\lambda_k\leq \frac{4(n+1)}{n^2k^2}\bigg(\sum_{i=1}^k \lambda_{i}^{\frac{1}{2}}\bigg)^2.
\end{equation*}
\begin{proof}
We infer from inequality (\ref{c3}) and the Chebyshev inequality
\begin{align*}
\bigg[\sum_{i=1}^k(\lambda_{k+1}&-\lambda_i)^2\bigg]^2\\
\leq\frac{4(n+1)}{n^2k^2}\bigg[\sum_{i=1}^k(\lambda_{k+1}-\lambda_i)\bigg]&\bigg[\sum_{i=1}^k(\lambda_{k+1}-\lambda_i)^2\bigg]\bigg[\sum_{i=1}^k\lambda^{\frac{1}{2}}\bigg]^2,
\end{align*}
or equivalently
\begin{align*}
\sum_{i=1}^k(\lambda_{k+1}-\lambda_i)^2\leq\frac{4(n+1)}{n^2k^2}\Big[\sum_{i=1}^k(\lambda_{k+1}-\lambda_i)\Big]\Big[\sum_{i=1}^k\lambda^{\frac{1}{2}}\Big]^2.
\end{align*}
Thus
\begin{align}\label{d3}
\sum_{i=1}^k(\lambda_{k+1}-\lambda_i)\bigg[(\lambda_{k+1}-\lambda_i)-\frac{4(n+1)}{n^2k^2}\Big[\sum_{i=1}^k\lambda_i^{\frac{1}{2}}\Big]^2\bigg]\leq
0.
\end{align}
Hence, since $\lambda_i\leq \lambda_k$, for all $i\leq k$, we can
easily deduce the inequality of Niu and Zhang from (\ref{d3}).
\end{proof}
\end{remark}
\subsection {The case when $\textit{\textbf{l}}\ge 3$}
We are now concerned with the problem (\ref{Kohn Lap}) for any $l \ge 3$. The result depends on the parity of $l$. In fact, we prove the following
\begin{theorem}\label{6th theorem}
 For any odd $l\geq3$, we have
\begin{align}
\sum_{i=1}^k(\lambda_{k+1}-\lambda_i)^2 & \leq \frac{1}{n}\bigg[\sum_{i=1}^k (\lambda_{k+1}-\lambda_i)\lambda_i^{\frac{1}{l}}\bigg]^{\frac{1}{2}} \times \nonumber\\
\label{stat a}\bigg\{\sum_{i=1}^k(\lambda_{k+1}&-\lambda_i)^{2}\bigg[\big(2l(n+l-1)\big)\lambda_i^{\frac{l-1}{l}}+c_1(n,l)\Big(\lambda_i+\lambda_i^{\frac{l-2}{l}}\Big) \bigg]\bigg\}^{\frac{1}{2}}
\end{align}
and for any even $l \ge 4$, we have
\begin{align}
\sum_{i=1}^k(\lambda_{k+1}-\lambda_i)^2 & \leq \frac{1}{n}\bigg[\sum_{i=1}^k (\lambda_{k+1}-\lambda_i) \lambda_i^{\frac{1}{l}}\bigg]^{\frac{1}{2}} \times \nonumber\\
\label{stat b}\bigg\{\sum_{i=1}^k(&\lambda_{k+1}-\lambda_i)^{2}\bigg[\big(2ln+4(l-1)\big)\lambda_i^{\frac{l-1}{l}}+c_2(n,l)\lambda_i^{\frac{l-1}{l}}\Big) \bigg]\bigg\}^{\frac{1}{2}},
\end{align}
where $c_1(n,l)$ and $c_2(n,l)$ are two constants depending on $n$ and $l$.
\end{theorem}
\begin{proof}
If we apply inequality (\ref{a}) with
$A=L^l=(-\Delta_{\mathbb{H}^n})^l$,
$B_1=x_1,\ldots,B_n=x_n,B_{n+1}=y_1,\ldots,B_{2n}=y_n$,
$T_1=X_1,\ldots,T_n=X_n,T_{n+1}=Y_1,\ldots,T_{2n}=Y_n$ and $f(x)=g(x)=(\lambda_{k+1}-x)^2$, then we obtain
\begin{align}
\bigg[\sum_{i=1}^k\sum_{p=1}^n(\lambda_{k+1}&-\lambda_i)^2\Big(\langle [X_p,x_p]u_i,u_i\rangle_{L^2}+\langle[Y_p,y_p]u_i,u_i\rangle_{L^2}\Big)\bigg]^2 \nonumber\\
\leq
4\bigg[\sum_{i=1}^k\sum_{p=1}^n(\lambda_{k+1}-&\lambda_i)^{2}\Big(\langle
[L^l,x_p]u_i,x_p u_i\rangle_{L^2}+ \langle [L^l,y_p]u_i,y_p
u_i\rangle_{L^2}\Big)\bigg]\times\nonumber\\
\bigg[\sum_{i=1}^k &\sum_{p=1}^n(\lambda_{k+1}-\lambda_i)\Big(\|X_p
u_i\|_{L^2}^2+\|Y_p u_i\|_{L^2}^2\Big)\bigg].\label{g1}
\end{align}
And as before, we have \begin{equation}\label{b3}
\langle[X_p,x_p]u_i,u_i\rangle_{L^2}=\langle[Y_p,y_p]u_i,u_i\rangle_{L^2}=1.
\end{equation}
On the other hand, to calculate $\sum_{p=1}^n \Big(\|X_p
u_i\|_{L^2}^2+\|Y_p u_i\|_{L^2}^2\Big)$, we need the following
result obtained by Niu and Zhang (see Lemma 2.3 in \cite{Niu-Zhang}) inspired by that of Chen and Qian \cite{ChenQian3} for the Laplacian:
\begin{lemma}\label{lemma1}
For any $d \geq 1$, we have
\begin{equation}
 \bigg(\int_{\Omega}\arrowvert \nabla_{\mathbb{H}^n}^d u_i\arrowvert^2\bigg)^{\frac{1}{d}} \leq \bigg(\int_{\Omega} \arrowvert \nabla_{\mathbb{H}^n}^{d+1}u_i\arrowvert^2 \bigg)^{\frac{1}{d+1}}
\end{equation}
where $\nabla^d=\begin{cases} L^{\frac{d}{2}} & \text{if $\hbox{d is
even}$,}\\
\nabla_{\mathbb{H}^n} L^{\frac{d-1}{2}} &\text{if $\hbox{d is
odd}$.}
\end{cases}$
\end{lemma}
And as a consequence (see Corollary 2.1 in \cite{Niu-Zhang}), we can easily obtain, for any $d \geq 1$ 
\begin{equation}
 \bigg(\int_{\Omega}\arrowvert \nabla_{\mathbb{H}^n} u_i\arrowvert^2\bigg) \leq \bigg(\int_{\Omega} \arrowvert \nabla^{d} _{\mathbb{H}^n}u_{i}\arrowvert^2 \bigg)^{\frac{1}{d}}.
\end{equation}
Therefore we have
\begin{equation}
\sum_{p=1}^n\bigg(\|X_p u_i\|_{L^2}^2+\|Y_p u_i\|_{L^2}^2\bigg)=\int_{\Omega}Lu_i.u_i \leq \bigg(\int_{\Omega} L^l u_i.u_i \bigg)^{\frac{1}{l}} =\lambda_i^{\frac{1}{l}}.\label{i2}
\end{equation}
Now we have to calculate
\begin{equation*}\sum_{p=1}^n\bigg(\langle
[L^l,x_p]u_i,x_p u_i\rangle_{L^2}+\langle [L^l,y_p]u_i,y_p
u_i\rangle_{L^2}\bigg).
\end{equation*}
For this purpose, we use the following lemma also obtained by Niu
and Zhang in \cite{Niu-Zhang}
\begin{lemma}\label{lemma2}
For any positive integer $d$, $1 \leq d \leq l$, we have
\begin{equation*}
L^d(x_pu_i)=x_p L^du_i - 2\sum_{q=1}^d L^{d-q}X_pL^{q-1}u_i,
\end{equation*}
$i=1,\ldots,k,\,p=1,\ldots,n$. This is also true for $y_p$ and
$Y_p$.
\end{lemma}
We infer, using Lemma \ref{lemma2},
\begin{equation*}
 [L^l,x_p]u_i  = L^l(x_p u_i) - x_p L^l u_i  = -2 \sum_{q=1}^l L^{l-q}X_p L^{q-1}u_i.
\end{equation*}
Therefore
\begin{align*}
\langle [L^l,x_p]u_i,x_p u_i \rangle_{L^2} & = -2 \sum_{q=1}^l \int_{\Omega} L^{l-q}X_p L^{q-1}u_i. x_p u_i \nonumber\\
& = -2 \sum_{q=1}^l \int_{\Omega} X_p L^{q-1}u_i.L^{l-q}(x_p u_i).
\end{align*}
The same identities hold with $y_p$ and $Y_p$. \\
Hence we obtain
\begin{align}
 \sum_{p=1}^n \bigg(\langle [L^l,x_p]u_i,x_p u_i \rangle_{L^2} + & \langle [L^l,y_p]u_i,y_p u_i \rangle_{L^2} \bigg)\nonumber\\
\label{j1}= -2\sum_{p=1}^n \sum_{q=1}^l \bigg(\int_{\Omega}X_p L^{q-1}u_i.L^{l-q}(x_p u_i) + & \int_{\Omega}Y_p L^{q-1}u_i.L^{l-q}(y_p u_i) \bigg).
\end{align}
Applying Lemma \ref{lemma2} once again to (\ref{j1}), we obtain
\begin{align}
\sum_{p=1}^n &\bigg(\langle [L^l,x_p]u_i,x_p u_i \rangle_{L^2} + \langle [L^l,y_p]u_i,y_p u_i \rangle_{L^2} \bigg)\nonumber\\
=-2\sum_{p=1}^n\sum_{q=1}^l\int_{\Omega}&x_pL^{l-q}u_i.X_p L^{q-1}u_i+4\sum_{p=1}^n\sum_{q=1}^{l-1}\int_{\Omega}X_pL^{l-q-1}u_i.X_p L^{q-1}u_i\nonumber\\
&+4\sum_{p=1}^n\sum_{q=1}^{l-2}\sum_{r=1}^{l-q-1}\int_{\Omega}L^{l-q-r}X_pL^{r-1}u_i.X_pL^{q-1}u_i\nonumber\\
-2\sum_{p=1}^n\sum_{q=1}^l\int_{\Omega}&y_pL^{l-q}u_i.Y_p
L^{q-1}u_i+4\sum_{p=1}^n\sum_{q=1}^{l-1}\int_{\Omega}Y_pL^{l-q-1}u_i.Y_p
L^{q-1}u_i\nonumber\\
&+4\sum_{p=1}^n\sum_{q=1}^{l-2}\sum_{r=1}^{l-q-1}\int_{\Omega}L^{l-q-r}Y_pL^{r-1}u_i.Y_pL^{q-1}u_i\label{k1}.
\end{align}
As in the proof of Theorem 5.1 in \cite{Niu-Zhang} (see the
calculation of the terms $I_2$ and $I'_2$), we can easily obtain, 
for any odd $l \ge 3$,
\begin{align}\label{w1}
\displaystyle\sum_{p=1}^n \bigg(\langle [L^l,x_p]u_i,x_p u_i
&\rangle_{L^2} + \langle [L^l,y_p]u_i,y_p u_i \rangle_{L^2} \bigg)\nonumber\\
\leq
\big(2l(n+l-1)\big)\lambda_i^{\frac{l-1}{l}}+&c_1(n,l)\bigg(\lambda_i+\lambda_i^{\frac{l-2}{l}}\bigg)
\end{align}
and for any even $l\ge 4$,
\begin{align}\label{y1}
\displaystyle\sum_{p=1}^n \bigg(\langle [L^l,x_p]u_i,x_p u_i
&\rangle_{L^2} + \langle [L^l,y_p]u_i,y_p u_i \rangle_{L^2} \bigg)\nonumber\\
\leq \bigg[\big(2ln+4(l-1)\big)+&c_2(n,l)\bigg]\lambda_i^{\frac{l-1}{l}}
\end{align}
where $c_1(n,3)=4$, \\
$\displaystyle
c_1(n,l)=2\sum_{q=1}^{l-2}\sum_{r=1}^{l-q-1}\bigg\{\sum_{\substack{s=
1 \\ s
\,odd}}^{l-q-r}\frac{2^snC_{l-q-r}^s}{(2n-1)^{\frac{s+1}{2}}}+\sum_{\substack{s=2 \\ s
\,even}}^{l-q-r}\frac{2^sC_{l-q-r}^s}{(2n-1)^{\frac{s}{2}}}\bigg\}$ for any odd $l \ge 5$
and  $\displaystyle
c_2(n,l)=4\sum_{q=1}^{l-2}\sum_{r=1}^{l-q-1}\bigg\{\sum_{\substack{s= 1
\\ s \,odd}}^{l-q-r}\frac{2^snC_{l-q-r}^s}{(2n-1)^{\frac{s+1}{2}}}
+\sum_{\substack{s= 0 \\ s
\,even}}^{l-q-r}\frac{2^sC_{l-q-r}^s}{(2n-1)^{\frac{s}{2}}}\bigg\}.$\\
Incorporating (\ref{b3}), (\ref{i2}) and (\ref{w1}) in (\ref{g1}), we obtain (\ref{stat a}). Similarly, to obtain (\ref{stat b}), we incorporate (\ref{b3}), (\ref{i2}) and (\ref{y1}) in (\ref{g1}).
\end{proof}
\begin{remark}
The inequality (\ref{stat a}) is not homogeneous in the eigenvalues $\lambda_i$ (i.e. it is not invariant under the change $L \rightarrow aL$, $\lambda_{i} \rightarrow a^{l}\lambda_{i}$ for $a>0$). Therefore, using inequality (\ref{stat a}) for $a^{l}L^l$, we obtain that for any $a>0$,
\begin{align}\label{p3}
\sum_{i=1}^k(\lambda_{k+1}-\lambda_i)^2 & \leq \frac{1}{n}\bigg[\sum_{i=1}^k (\lambda_{k+1}-\lambda_i)\lambda_i^{\frac{1}{l}}\bigg]^{\frac{1}{2}} \times \nonumber\\
\bigg\{\sum_{i=1}^k(\lambda_{k+1}&-\lambda_i)^{2}\bigg[2l(n+l-1)\lambda_i^{\frac{l-1}{l}}+c_1(n,l)\Big(a\lambda_i+\frac{1}{a}\lambda_i^{\frac{l-2}{l}}\Big) \bigg]\bigg\}^{\frac{1}{2}},
\end{align}
for any odd $l \ge 3$.\\
Optimising with respect to $a$, we find the following improvement of the inequality (\ref{stat a}) 
\begin{align}\label{r3}
\sum_{i=1}^k(\lambda_{k+1}-\lambda_i)^2  \leq \frac{1}{n}\Big(2l(n+l-1)+c_1(n,l)\Big)^{\frac{1}{2}}&\bigg[\sum_{i=1}^k (\lambda_{k+1}-\lambda_i)\lambda_i^{\frac{1}{l}}\bigg]^{\frac{1}{2}} \times \nonumber\\
\bigg[\sum_{i=1}^k(\lambda_{k+1}&-\lambda_i)^{2}\lambda_i^{\frac{l-1}{l}}\bigg]^{\frac{1}{2}}
\end{align}
which is homogeneous on the eigenvalues $\lambda_i$.
\end{remark}
As for the case when $l=2$, we can deduce inequalities of Yang-type
for $l\ge 3$.
\begin{corollary}
We have, for any odd $l\ge 3$,
\begin{align}\label{cor2}
{}\sum_{i=1}^k&(\lambda_{k+1}-\lambda_i)^2 \\ 
\leq\frac{1}{n^2}\sum_{i=1}^k(\lambda_{k+1}-\lambda_i)&\nonumber\bigg[\big(2l(n+l-1)\big)\lambda_i+c_1(n,l)\Big(\lambda_i^{\frac{l+1}{l}}+\lambda_i^{\frac{l-1}{l}}\Big)\bigg]
\end{align}
and for any even $l\ge 4$
\begin{equation}\label{cor3}
 \sum_{i=1}^k(\lambda_{k+1}-\lambda_i)^2\leq \frac{2ln+4(l-1)+c_2(n,l)}{n^2}\sum_{i=1}^k(\lambda_{k+1}-\lambda_i)\lambda_i.
\end{equation}
where $c_{1}(n,l)$ and $c_{2}(n,l)$ are explicit constants depending only on $n$ and $l$.
\end{corollary}
\begin{proof}
Applying Lemma \ref{lemma3} with $A_i=\lambda_{k+1}-\lambda_i$,
$B_i=2l(n+l-1)\lambda_i^{\frac{l-1}{l}}+c_1(n,l)(\lambda_i+\lambda_i^{\frac{l-2}{l}})$
and $C_i=\lambda_i^{\frac{1}{l}}$, we obtain
\begin{align}\label{lem1}
\bigg[\sum_{i=1}^k(\lambda_{k+1}-\lambda_i)\lambda_i^{\frac{1}{l}}\bigg]&\bigg\{\sum_{i=1}^k(\lambda_{k+1}-\lambda_i)^{2}\bigg[2l(n+l-1)\lambda_i^{\frac{l-1}{l}}\nonumber\\
+c_1(n,l)\bigg(\lambda_i&+\lambda_i^{\frac{l-2}{l}}\bigg)\bigg]\bigg\}\leq \sum_{i=1}^k(\lambda_{k+1}-\lambda_i)^2\times\nonumber\\
\sum_{i=1}^k(\lambda_{k+1}-\lambda_i)&\lambda_i^{\frac{1}{l}}\bigg[2l(n+l-1)\lambda_i^{\frac{l-1}{l}}+c_1(n,l)\bigg(\lambda_i+\lambda_i^{\frac{l-2}{l}}\bigg)\bigg]\nonumber\\
=\sum_{i=1}^k(\lambda_{k+1}-\lambda_i)^2\sum_{i=1}^k&(\lambda_{k+1}-\lambda_i)\bigg[2l(n+l-1)\lambda_i+c_1(n,l)\bigg(\lambda_i^{\frac{l+1}{l}}+\lambda_i^{\frac{l-1}{l}}\bigg)\bigg]
\end{align}
Inequality (\ref{cor2}) can be deduced from (\ref{stat a}) and
(\ref{lem1}), for any odd $l\ge 3$.\\We proceed in the same way to
obtain (\ref{cor3}), i.e. applying Lemma \ref{lemma3} but with
$A_i=\lambda_{k+1}-\lambda_i$,
$B_i=\Big(2ln+4(l-1)+c_2(n,l)\Big)\lambda_i^{\frac{l-1}{l}}$ and
$C_i=\lambda_i^{\frac{1}{l}}$.
\end{proof}
\begin{remark}
Inequalities (\ref{stat a}) and (\ref{stat b}) are sharper than the following inequalities, proved by Niu and Zhang (\cite{Niu-Zhang}),
\begin{equation}\label{ineq 1}
 \lambda_{k+1}-\lambda_k \leq \frac{\sum_{i=1}^k\lambda_i^{\frac{1}{l}}}{n^2k^2}\bigg[\big(2l(n+l-1)\big)\sum_{i=1}^k\lambda_i^{\frac{l-1}{l}}+c_1(n,l)\sum_{i=1}^k\bigg(\lambda_i+\lambda_i^{\frac{l-2}{l}}\bigg)\bigg]
\end{equation}
if $l\ge 3$ is odd and
\begin{equation}\label{ineq 2}
\lambda_{k+1}-\lambda_k \leq \frac{\sum_{i=1}^k\lambda_i^{\frac{1}{l}}}{n^2k^2}\bigg[\big(2ln+4(l-1)\big)\sum_{i=1}^k\lambda_i^{\frac{l-1}{l}}+c_2(n,l)\sum_{i=1}^k \lambda_i^{\frac{l-1}{l}}\bigg],
\end{equation}
if $l \ge 4$ is even,\\ 
$c_1(n,l)$ and $c_2(n,l)$ are as in the proof of Theorem \ref{6th theorem}.
\end{remark}
\begin{proof}
 By the Chebyshev inequality, we infer from (\ref{stat a}), for any odd $l\ge 3$
\begin{align*}
\bigg[\sum_{i=1}^k &(\lambda_{k+1}-\lambda_i)^2\bigg]^2\leq \frac{1}{n^2k^2}\bigg[\sum_{i=1}^k (\lambda_{k+1}-\lambda_i)\bigg]\bigg[\sum_{i=1}^k (\lambda_{k+1}-\lambda_i)^2\bigg] \times\\
&
\bigg[\sum_{i=1}^k\lambda_i^{\frac{1}{l}}\bigg]\bigg[2l(n+l-1)\sum_{i=1}^k\lambda_i^{\frac{l-1}{l}}+c_1(n,l)\sum_{i=1}^k\bigg(\lambda_i+\lambda_i^{\frac{l-2}{l}}\bigg)\bigg]
\end{align*}
or equivalently
\begin{align*}
\sum_{i=1}^k &(\lambda_{k+1}-\lambda_i)^2 \leq \frac{1}{n^2k^2}\bigg[\sum_{i=1}^k (\lambda_{k+1}-\lambda_i)\bigg]\bigg[\sum_{i=1}^k\lambda_i^{\frac{1}{l}}\bigg]\times\\
&\bigg[2l(n+l-1)\sum_{i=1}^k\lambda_i^{\frac{l-1}{l}}+c_1(n,l)\sum_{i=1}^k\bigg(\lambda_i+\lambda_i^{\frac{l-2}{l}}\bigg)\bigg].
\end{align*}
Thus \begin{align}\label{l3}
\sum_{i=1}^k(\lambda_{k+1}-&\lambda_i)\Bigg[\lambda_{k+1}-\lambda_i-\frac{1}{n^2k^2}\bigg(\sum_{i=1}^k\lambda_i^{\frac{1}{l}}\bigg)\times\nonumber\\
&\bigg(2l(n+l-1)\sum_{i=1}^k\lambda_i^{\frac{l-1}{l}}+c_1(n,l)\sum_{i=1}^k\Big(\lambda_i+\lambda_i^{\frac{l-2}{l}}\Big)\bigg)\Bigg]\leq 0
     \end{align}
which implies (\ref{ineq 1}), since $\lambda_i \leq \lambda_{k}$ for
$i \le k$.\\ Similarly, we prove that inequality (\ref{stat b}) is
sharper than (\ref{ineq 2}).
\end{proof}

\subsection*{Acknowledgments}
We thank the referee for the suggestions which allowed us to improve the first version of the paper.\\
This work was partially supported by the ANR (Agence Nationale de la Recherche) through FOG project(ANR-07-BLAN-0251-01).

\end{document}